\documentclass[12pt,A4,reqno]{amsart}
\usepackage{amsfonts}
\usepackage{mathrsfs}

\usepackage{amssymb}
\usepackage{amsmath,amscd}
\usepackage{color}
\usepackage{epsf}
\usepackage{graphicx}
\usepackage{xypic}

\theoremstyle{plain}
\newtheorem{thm}{Theorem}[section]
\newtheorem{lem}[thm]{Lemma}
\newtheorem{prop}[thm]{Proposition}
\newtheorem{cor}[thm]{Corollary}

\theoremstyle{definition}

\newtheorem{rem}[thm]{Remark}
\newtheorem{exam}{Example}

\newcommand{\Q}{\mathbb Q}

\newcommand{\Z}{\mathbb Z}
\newcommand{\F}{\mathbb F}

\begin{document}

\title [\ ] {Small Cover and Halperin-Carlsson Conjecture}

\author{Li Yu}
\address{Department of Mathematics and IMS, Nanjing University, Nanjing, 210093, P.R.China
  \newline
     \textit{and}
 \newline
    \qquad  Department of Mathematics, Osaka City University, Sugimoto,
     Sumiyoshi-Ku, Osaka, 558-8585, Japan}

 \email{yuli@nju.edu.cn}

\thanks{This work is partially supported by
the Scientific Research Foundation for the Returned Overseas Chinese
Scholars, State Education Ministry and by the Japanese Society for
the Promotion of Sciences (JSPS grant no. P10018).}

%\date{July 21, 2008}



 \begin{abstract}
   We prove that the Halperin-Carlsson conjecture holds for
    any free $(\Z_2)^m$-action on a compact manifold whose orbit space is a
     small cover. In addition, we show that if the total space of a principal
     $(\Z_2)^m$-bundle over a small cover is connected, it must
     be equivalent to a partial quotient of the corresponding
     real moment-angle manifold with some canonical $\Z_2$-torus action.
 \end{abstract}

\maketitle

  \section{Introduction}

     For any prime $p$, let $\Z_p$ denote the quotient group $\Z\slash
    p\Z$. And let $S^1$ be the circle group.\vskip .2cm

    \textbf{Halperin-Carlsson Conjecture:} If $G =(\Z_p)^m$ ($p$ is
    a prime) or $(S^1)^m$ acts freely on a finite dimensional CW-complex $X$,
    then $\sum^{\infty}_{i=0} \dim_{\Z_p}
    H^i(X, \Z_p)\geq 2^m$ or $\sum^{\infty}_{i=0} \dim_{\Q}
    H^i(X, \Q)\geq 2^m$ respectively.\vskip .2cm

       The above conjecture was proposed in the middle of 1980s by S.~Halperin
       in~\cite{Halperin85} for the torus case, and by G.~Carlsson
       in~\cite{Cal85} for the $\Z_p$-torus case.
       It is also called \textit{toral rank conjecture} in some papers.
       \vskip .2cm

   In the earlier time, this conjecture mainly took the form of
   whether a free $(\Z_p)^m$-action on a product of spheres
   $S^{n_1}\times \cdots \times S^{n_k}$ implies $m\leq k$.
   Many authors have studied this intriguing conjecture
   and contributed results with respect to different aspects (see~\cite{Corner57}--\cite{Hanke09}).
   The
   reader is referred to see a survey of such results in~\cite{Adem04} and ~\cite{Allday93}.
    But the general case is
     still open for any prime $p$.\vskip .2cm

    For general finite dimensional CW-complexes, the conjecture has been proved in~\cite{Puppe09} for
    $m\leq 3$ in the torus and $\Z_2$-torus
    cases and $m\leq 2$ in the odd $\Z_p$-torus case.
       More recently, Cao and L\"u (see~\cite{LuZhi09}) and
     Ustinovsky (see~\cite{Uto09}) independently proved the
     following result, which confirmed the
     Halperin-Carlsson conjecture for some canonical
     $\Z_2$-torus actions on real moment-angle complexes. \vskip .2cm

     \begin{thm}[see~\cite{LuZhi09} and~\cite{Uto09}] \label{thm:CaoLU}
      If $K^{n-1}$ is an $(n-1)$-dimensional simplicial complex
      on the vertex set $[ d\, ]$. Then the real moment-angle complex
       $\mathcal{Z}_{K^{n-1}}$ over $K^{n-1}$ must satisfy:
     $ \sum_{i} \dim_{\Z_2} H^i(\mathcal{Z}_{K^{n-1}}, \Z_2)\geq
     2^{d-n} $. In
     particular, if $P^n$ is an $n$-dimensional simple polytope with $d$ facets,
     then the real moment-angle manifold $\mathcal{Z}_{P^n}$ must satisfy:
     $ \sum_i \dim_{\Z_2} H^i(\mathcal{Z}_{P^n}, \Z_2)\geq
     2^{d-n}$.
     \end{thm}
     \vskip .2cm

    \begin{rem}
       Indeed, much stronger results were obtained in~\cite{LuZhi09}
       and~\cite{Uto09}. For example, it was shown in~\cite{LuZhi09}
       and~\cite{Uto09} that the Theorem~\ref{thm:CaoLU} holds even if
       the $\Z_2$-coefficient is replaced by the rational coefficient.
    \end{rem}
      \vskip .2cm

  \begin{rem}
    There is a purely algebraic analogue of the Halperin-Carlsson conjecture, which is proposed
    in~\cite{Cal85} in the context of commutative algebras. Some related results were
     obtained in~\cite{Cal87}.
     \end{rem}
  \vskip .2cm

  In this paper, we will only study the conjecture for
       $G=(\Z_2)^m$ and $X$ being a closed manifold.
       In addition, we will use the following conventions: \vskip .1cm
       \begin{enumerate}
         \item we always treat $(\Z_2)^m$ as an additive group; \vskip .1cm
        \item any manifold and submanifold in this paper are smooth; \vskip .1cm
         \item we do not distinguish an embedded submanifold and its image. \vskip .2cm
       \end{enumerate}

     Suppose $(\Z_2)^m$ acts freely and smoothly on a closed $n$-manifold
     $M^n$. Let $Q^n =M^n \slash (\Z_2)^m$ be the orbit
     space. Then $Q^n$ is a closed $n$-manifold too.
     Let $\pi : M^n \rightarrow Q^n$ be the orbit map.
     We can think of $M^n$ either as a principal $(\Z_2)^m$-bundle over $Q^n$ or
     as a regular covering over $Q^n$ with deck transformation
     group $(\Z_2)^m$.
     In algebraic topology, we have a standard way to
      recover $M^n$ from $Q^n$ using the universal covering space
      of $Q^n$ and the monodromy of the covering (see~\cite{Hatcher02}).
      However, it is not very easy for us to
        visualize the total space of the covering in this approach.
       In ~\cite{Yu2010}, a new
       way of constructing principal $(\Z_2)^m$-bundles over
       closed manifolds is introduced, which allows us to visualize
       this kind of objects more easily.\vskip .2cm

      Indeed, it is shown in~\cite{Yu2010} that
      $\pi: M^n \rightarrow Q^n$ determines a $(\Z_2)^m$-coloring $\lambda_{\pi}$ on a nice
      manifold with corners $V^n$ (called a \textit{$\Z_2$-core} of
      $Q^n$), and up to equivariant homeomorphism, we can recover $M^n$
     by a standard \textit{glue-back construction}
     from $V^n$ and $\lambda_{\pi}$.
     Using this new language, we will prove the following theorem which
     confirms the Halperin-Carlsson conjecture in some new cases.
     \vskip .2cm

       \begin{thm} \label{thm:main}
           Suppose $(\Z_2)^m$ acts freely on
           a closed $n$-manifold $M^n$ whose orbit space is homeomorphic to a small
           cover, then we must have:
           \begin{equation} \label{Equ:Main}
               \sum_i \dim_{\Z_2} H^i(M^n, \Z_2)\geq 2^m
           \end{equation}
       \end{thm}

        Recall that an $n$-dimensional small cover is a
        closed $n$-manifold with a locally standard $(\Z_2)^n$-action whose orbit space is a
    simple polytope (see~\cite{DaJan91}).\vskip .2cm

    Suppose $P^n$ is a simple polytope with $d$ facets.
    There is a \textit{canonical action} of
       $(\Z_2)^d$ on the real moment-angle complex $\mathcal{Z}_{P^n}$
        whose orbit space is $P^n$. For a subtorus $H \subset (\Z_2)^{d}$, if $H$
       acts freely on $\mathcal{Z}_{P^n}$ through the canonical action,
       $\mathcal{Z}_{P^n} \slash H$ is called a \textit{partial quotient}
       of $\mathcal{Z}_{P^n}$ (see~\cite{BP02}). In addition, if
       there is a subgroup $\widetilde{H}$ of $(\Z_2)^d$ with $\widetilde{H} \supset H$ and
       $\widetilde{H}$ also acts freely on $\mathcal{Z}_{P^n}$ through the canonical action,
       we will get an induced free
       action of $\widetilde{H}\slash H$ on $\mathcal{Z}_{P^n} \slash H$ whose orbit space
       is $\mathcal{Z}_{P^n} \slash \widetilde{H}$. By abusing of terminology,
       we also call this kine of $\widetilde{H}\slash H$-action
        on $\mathcal{Z}_{P^n} \slash H$ \textit{canonical}.
        \vskip .2cm

       In addition,
       two principal $(\Z_2)^m$-bundles $M^n_1$ and $M^n_2$ over a space $Q^n$ are called
    \textit{equivalent} if there is a homeomorphism
   $f: M^n_1 \rightarrow M^n_2$ together with a group automorphism
   $\sigma : (\Z_2)^m \rightarrow (\Z_2)^m$ such that:
   \begin{enumerate}
     \item  $f(g\cdot x) = \sigma(g)\cdot f(x)$ for all $g\in (\Z_2)^m$ and $x \in M^n_1$,
     and \vskip .1cm
     \item  $f$ induces the identity map on the orbit space.
   \end{enumerate}
   Under these conditions, we also say that the \textit{free $(\Z_2)^m$-actions} on $M^n_1$ and $M^n_2$
   \textit{are equivalent}. \vskip .2cm

    We can prove
       the following proposition as a by-product of our discussion.\vskip .2cm

       \begin{prop} \label{prop:Partial-Quotient}
            Suppose $Q^n$ is a small cover over a simple polytope
            $P^n$ and $M^n$ is a principal $(\Z_2)^m$-bundle over $Q^n$. If
             $M^n$ is connected, then $M^n$ must be equivalent to
            a partial quotient $\mathcal{Z}_{P^n} \slash H$
             as principal $(\Z_2)^m$-bundles over $Q^n$.
         \end{prop}

   \vskip .2cm

          The paper is organized as follows.
          In section~\ref{Sec2}, we will briefly review the
          $\Z_2$-core of a manifold and the glue-back construction
          introduced in~\cite{Yu2010} and study some topological aspects of
          the glue-back construction. Then in section~\ref{Sec3},
         we will give a proof of Theorem~\ref{thm:main}. In section~\ref{Sec4}, we will study
          real moment-angle manifolds from the viewpoint of glue-back construction and
         give a proof of Proposition~\ref{prop:Partial-Quotient}. \\

 \section{Glue-back Construction} \label{Sec2}

       Suppose $(\Z_2)^m$ acts freely and smoothly on an $n$-dimensional closed manifold
     $M^n$. Then the orbit space $Q^n = M^n \slash (\Z_2)^m$ is naturally a closed manifold.
     In the rest of this section, we always assume that $Q^n$ is connected and $H^1(Q^n,\Z_2) \neq 0$.
        Indeed, if $Q^n$ is not connected, we can just apply our
      discussion to each connected component of $Q^n$.
      And if $H^1(Q^n,\Z_2)=0$, $M^n$ must
     be homeomorphic to $Q^n\times (\Z_2)^m$.
     \vskip .2cm

    Let $\pi: M^n \rightarrow Q^n$ be the orbit map. If we think of
    $M^n$ as a principal $(\Z_2)^m$-bundle over $Q^n$, it is
    classified by an element $ \Lambda_{\pi} \in H^1(Q^n,(\Z_2)^m)$. To recover
    the $M^n$ from $Q^n$, we shall construct a manifold with corners
    from $Q^n$ that can carry the information of $\Lambda_{\pi}$. This is done
    in the following way (see~\cite{Yu2010}).\vskip .2cm

     By a standard argument of
     the intersection theory in differential topology, we can show that there exists
     a collection of $(n-1)$-dimensional
     compact embedded
     submanifolds $\Sigma_1,\cdots,\Sigma_k$
     in $Q^n$ such that their homology classes $\{ [\Sigma_1],\cdots,[\Sigma_k] \}$
    form a basis of $H_{n-1}(Q^n, \Z_2)\cong H^1(Q^n,\Z_2)\neq 0$.
    Moreover, we can put $\Sigma_1,\cdots,\Sigma_k$
    in \textit{general position} in $Q^n$, which means: \vskip .2cm
      \begin{enumerate}
        \item $\Sigma_1, \cdots, \Sigma_k$ intersect transversely with each
      other, and\vskip .2cm

        \item if $\Sigma_{i_1} \cap \cdots \cap \Sigma_{i_s}$ is not empty,
               then it is an embedded submanifold
               of $Q^n$ with codimension $s$.\vskip .2cm
      \end{enumerate}

      Then we cut $Q^n$ open along $\Sigma_1,\cdots,\Sigma_k$,
       i.e. we choose a small tubular neighborhood $N(\Sigma_i)$ of each $\Sigma_i$
        and remove the interior of each $N(\Sigma_i)$ from $Q^n$.
        Then we get a nice manifold with corners $V^n = Q^n - \bigcup_i
         int(N(\Sigma_i))$, which is called a \textit{$\Z_2$-core} of $Q^n$
        from cutting $Q^n$ open along $\Sigma_1,\cdots,\Sigma_k$
        (see Figure~\ref{p:Panel_Struc} for example).
        Recall that a manifold with corners is called \textit{nice} if
     each codimension $l$ face of the manifold is the intersection of exactly $l$
     facets (see~\cite{Janich68} and ~\cite{Da83}).
         Here, the niceness
        of $V^n$ follows from the general position of $\Sigma_1,\cdots,\Sigma_k$
         in $Q^n$.
        The boundary of $N(\Sigma_i)$ is called the \textit{cut section}
        of $\Sigma_i$ in $Q^n$, and $\{ \Sigma_1,\cdots,\Sigma_k \}$ is
        called a \textit{$\Z_2$-cut system} of $Q^n$.
        Moreover, we can choose each $\Sigma_i$ to be connected.
        \vskip .2cm

         Notice that the projection $\eta_i : \partial N(\Sigma_i) \rightarrow
        \Sigma_i$ is a double cover, either trivial or nontrivial.
        Let $\overline{\tau}_i$ be the generator of
       the deck transformation of $\eta_i$.
       Then $\overline{\tau}_i$ is a free involution on $\partial N(\Sigma_i)$, i.e.
        $\overline{\tau}_i$ is a homeomorphism with no fixed point and $\overline{\tau}^2_i =
        id$. By applying some local deformations to these
        $\overline{\tau}_i$ if necessary (see~\cite{Yu2010}),
        we can construct an \textit{involutive panel structure} on $\partial V^n$,
         which means that
         the boundary of $V^n$ is the union of some compact subsets
         $P_1,\cdots, P_k$ (called \textit{panels}) that satisfy the following three
         conditions:\vskip .2cm

           \begin{itemize}
             \item[(a)] each panel $P_i$ is a disjoint union of facets of $V^n$ and
               each facet is contained in exactly one panel; \vskip .1cm

             \item[(b)] there exists a free involution
                  $\tau_i$ on each $P_i$ which
                  sends a face $f \subset P_i$ to a face $f'\subset P_i$
                 (it is possible that $f'=f$); \vskip .1cm

             \item[(c)] for $\forall\, i \neq j$, $\tau_i (P_i \cap P_j) \subset P_i \cap
             P_j$ and $\tau_i\circ\tau_j = \tau_j\circ \tau_i
                 : P_i \cap P_j \rightarrow P_i \cap P_j$.
            \end{itemize}

         Indeed, the $P_i$ above consists of those facets of $V^n$ that lie in the cut section
         of $\Sigma_i$ and
          $\tau_i: P_i \rightarrow P_i$ is the restriction of the
          modified $\overline{\tau}_i$ to $P_i$ (see~\cite{Yu2010} for
          the details of these constructions). \vskip .3cm

  \begin{rem}
    A more general notion of \textit{involutive panel structure} is
    introduced in~\cite{Yu2010} where the involution $\tau_i$ in (b) above
    is not required to be free.
    This general notion is used in~\cite{Yu2010} to unify the
    construction of all locally standard $(\Z_2)^m$-actions on closed
    manifolds from the orbit spaces.
  \end{rem}
  \vskip .2cm
          \begin{figure}
            \includegraphics[width=0.38\textwidth]{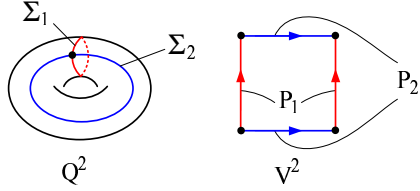}
               \caption{A $\Z_2$-core of torus}
               \label{p:Panel_Struc}
            \end{figure}

       Let $\mathcal{P}(V^n) = \{ P_1,\cdots, P_k \}$
       denote the set of all panels in $V^n$.
         Any map $\lambda: \mathcal{P}(V^n) \rightarrow
          (\Z_2)^m$ is called a \textit{$(\Z_2)^m$-coloring on
          $V^n$}, and any element in $(\Z_2)^m$ is called a \textit{color}.
          \\

      Now, let us see how to recover the principal $(\Z_2)^m$-bundle
       $\pi: M^n \rightarrow Q^n$ from the $\Z_2$-core $V^n$ of
       $Q^n$ and the element
            \begin{equation} \label{Equ:Classify-Bundle}
          \Lambda_{\pi} \in H^1(Q^n,(\Z_2)^m) \cong
          \mathrm{Hom}(H_1(Q^n,\Z_2),(\Z_2)^m).
            \end{equation}
       By the Poincar\'e duality for closed manifolds, there is a group isomorphism
       $$\kappa : H_{n-1}(Q^n, \Z_2) \rightarrow
       H_1(Q^n,\Z_2).$$
        So we can assign an element of $(\Z_2)^m$ to each panel $P_i$ of $V^n$ by:
         \[ \lambda_{\pi}(P_i) =  \Lambda_{\pi}(\kappa([\Sigma_i])) \in (\Z_2)^m \]
         We call $\lambda_{\pi}$ the \textit{associated $(\Z_2)^m$-coloring} of
         $\pi: M^n \rightarrow Q^n$ on $V^n$.
        \vskip .2cm

          Generally, for any $(\Z_2)^m$-coloring $\lambda$
          on $V^n$, we can glue $2^m$ copy of $V^n$ by:
         \begin{equation} \label{Glue_Back}
            M(V^n, \{ P_i, \tau_i \},\lambda) := V^n \times (\Z_2)^m \slash \sim
         \end{equation}
           Where $(x,g)\sim (x',g') $ whenever
            $x' = \tau_i(x)$ for some $P_i$ and
              $g' = g+ \lambda(P_i) \in (\Z_2)^m$. \vskip .2cm

          Note that if $x$ is in the relative interior of $P_{i_1} \cap \cdots \cap
             P_{i_s}$,
           $(x,g) \sim (x',g')$ if and only if
           $ (x',g')= ( \tau^{\varepsilon_s}_{i_s}\circ \cdots \circ
           \tau^{\varepsilon_1}_{i_1}(x), g + \varepsilon_1\lambda(P_1) + \cdots +
           \varepsilon_s\lambda(P_s))$
           where $\varepsilon_j= 0$ or $1$ for any $1\leq  j \leq s$ and
           $\tau^{0}_{i_j} := id$. \vskip .2cm

           $M(V^n,\{ P_i, \tau_i \},\lambda)$ is called the
          \textit{glue-back construction} from $(V^n,\lambda)$.
            Also, we use $M(V^n,\lambda)$
            to denote $M(V^n,\{ P_i, \tau_i \},\lambda)$
            if there is no ambivalence about the involutive panel structure on $V^n$
            in the context.\vskip .2cm

           In addition, let $[(x,g)]\in M(V^n,\lambda)$ denote the equivalence class
           of $(x,g)$ defined in~\eqref{Glue_Back}.
           It is shown in~\cite{Yu2010} that $M(V^n,\lambda)$
           is a closed manifold
           with a smooth free $(\Z_2)^m$-action defined by:
            \begin{equation} \label{Equ:FreeAction}
              g\cdot [(x,g_0)] := [(x, g+g_0)],\; \forall\, x\in V^n, \
             \forall\, g, g_0\in (\Z_2)^m.
           \end{equation}
          And the orbit space
            of $M(V^n,\lambda)$ under
          this free $(\Z_2)^m$-action is homeomorphic to $Q^n$.
          We say~\eqref{Equ:FreeAction} defines the \textit{natural $(\Z_2)^m$-action}
          on $M(V^n,\lambda)$.
          In this paper, we always associate this natural free
          $(\Z_2)^m$-action to $M(V^n,\lambda)$. Moreover, for any
          subgroup $N \subset (\Z_2)^m$, the induced
          action of $(\Z_2)^m\slash N$ on $M(V^n,\lambda)\slash N$
          from the natural action
          is also free and its orbit space is homeomorphic to $M(V^n,\lambda) \slash (\Z_2)^m =
          Q^n$. By abusing of terminology, we also call this
          $(\Z_2)^m\slash N$-action on $M(V^n,\lambda)\slash N$ \textit{natural}.
           \vskip .2cm

        \begin{thm}[Yu~\cite{Yu2010}] \label{Thm:Equiv-Isom}
          For any principal $(\Z_2)^m$-bundle $\pi : M^n \rightarrow
          Q^n$, let $\lambda_{\pi}$ be the associated $(\Z_2)^m$-coloring on $V^n$.
          Then $M(V^n, \lambda_{\pi})$ and $M^n$ are equivalent as principal
          $(\Z_2)^m$-bundles over $Q^n$.
        \end{thm}
          \vskip .2cm

       \begin{exam} \label{Exam:Covering}
           Figure~\ref{p:Torus-Cover} shows two principal
         $(\Z_2)^2$-bundles over $T^2$ via glue-back constructions
        from two different $(\Z_2)^2$-colorings on a $\Z_2$-core of
        $T^2$. The $\{ e_1 ,e_2 \}$ in the picture is a
        linear basis of $(\Z_2)^2$. The first $(\Z_2)^2$-coloring gives a torus, and the
        second one gives a disjoint union of two tori. In addition,
       there is a double covering
        $\eta$ (defined in~\eqref{Equ:Involution} later)
        from the torus on the top to either one of the torus below it.
     \vskip .2cm

   \begin{figure}
        \begin{equation*}
        \vcenter{
            \hbox{
                  \mbox{$\includegraphics[width=0.5\textwidth]{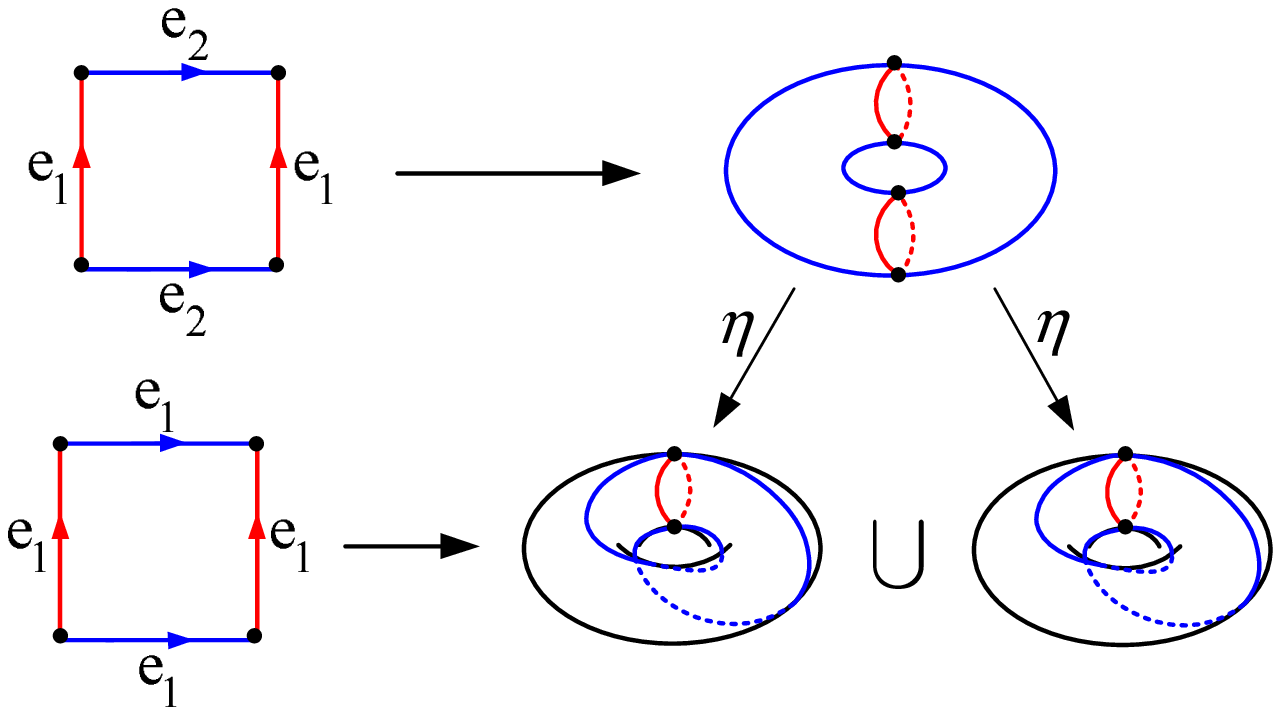}$}
                 }
           }
     \end{equation*}
   \caption{
       } \label{p:Torus-Cover}
   \end{figure}

          \end{exam}

         \begin{exam}
            Figure~\ref{p:Klein-Bottle} shows a $\Z_2$-core of the Klein bottle
            with three different $\Z_2$-colorings, where $\Z_2= \langle a \rangle$.
            So from the glue-back construction, we get three
            inequivalent double coverings of the Klein bottle.
              From left to right in Figure~\ref{p:Klein-Bottle},
             the first
            $\Z_2$-coloring gives a torus,
            while the second and the third both give the Klein bottle. \vskip .2cm

   \begin{figure}
        \begin{equation*}
        \vcenter{
            \hbox{
                  \mbox{$\includegraphics[width=0.58\textwidth]{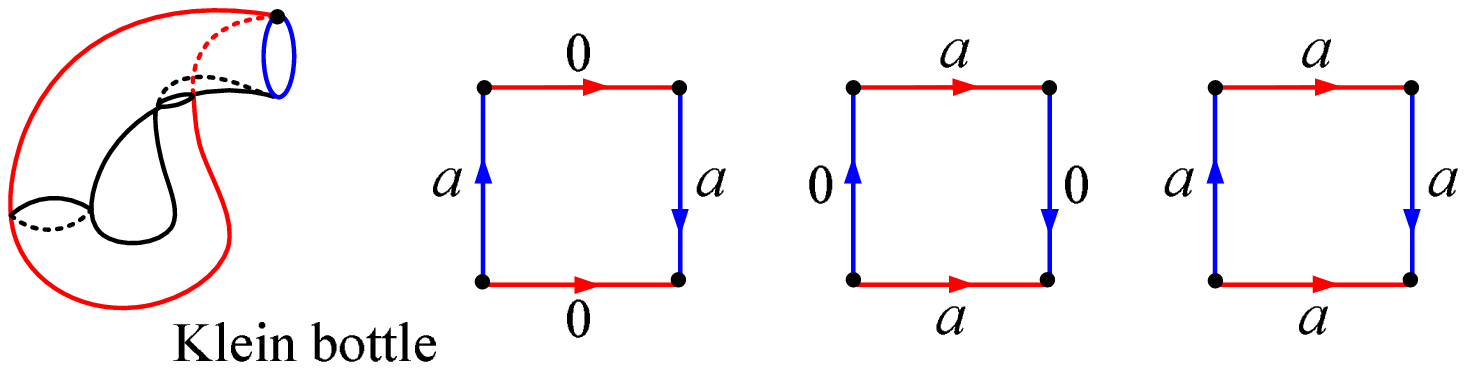}$}
                 }
           }
     \end{equation*}
   \caption{
       } \label{p:Klein-Bottle}
   \end{figure}

         \end{exam}

          For any integer $m\geq 1$, we define
      \begin{align*}
       \quad  \mathrm{Col}_m(V^n) & := \text{the set of all $(\Z_2)^m$-colorings
       on $V^n$} \\
         &\ = \{ \lambda\ | \ \lambda: \mathcal{P}(V^n) \rightarrow (\Z_2)^m
         \}, \\
         L_{\lambda} &:= \text{the subgroup of $(\Z_2)^m$ generated
         by $ \{ \lambda(P_1), \cdots ,\lambda(P_k) \} $},\\
          \text{rank}(\lambda)  & :=  \mathrm{dim}_{\Z_2}
          L_{\lambda}, \ \ \forall\, \lambda \in \mathrm{Col}_m(V^n)
          \qquad\qquad
       \end{align*}

      \vskip .2cm

  \begin{thm}[Yu~\cite{Yu2010}] \label{thm:comp}
        For any $(\Z_2)^m$-coloring $\lambda$ on the panels of $V^n$,
       $M(V^n,\lambda)$ has $2^{m-\mathrm{rank}(\lambda)}$ connected
       components which are pairwise homeomorphic. Let $\theta_{\lambda}:
       V^n\times (\Z_2)^m \rightarrow M(V^n,\lambda)$ be the
       quotient map. Then each
       connected component of $M(V^n,\lambda)$ is homeomorphic to
       $\theta_{\lambda}(V^n\times L_{\lambda})$, and
       there is a free action of $L_{\lambda} \cong (\Z_2)^{\mathrm{rank}(\lambda)}$ on
       each connected component of $M(V^n,\lambda)$ whose orbit space is $Q^n$.
   \end{thm}

    In addition, $\lambda \in \mathrm{Col}_m(V^n)$ is called \textit{maximally independent}
      if $\text{rank}(\lambda)= k = \dim_{\Z_2} H_{n-1}(Q^n, \Z_2)$.
      Note that if $\lambda \in \mathrm{Col}_m(V^n)$ is maximally
      independent, we must have $m\geq k$.
      \vskip .2cm

     \begin{lem} \label{Lem:Max_Homeo}
        For any $m\geq \dim_{\Z_2} H_{n-1}(Q^n, \Z_2)$,
        if $\lambda_1, \lambda_2 \in \mathrm{Col}_m(V^n)$
          are both maximally
       independent, $M(V^n,\lambda_1)$ must be equivalent to $M(V^n,\lambda_2)$
       as principal $(\Z_2)^m$-bundles over $Q^n$.
       \end{lem}
       \begin{proof}
         Since $\lambda_1$ and $\lambda_2$ are both maximally independent,
         $\{ \lambda_1(P_1), \cdots ,\lambda_1(P_k) \}$ and $\{ \lambda_2(P_1), \cdots ,\lambda_2(P_k) \}$
         are linearly independent subsets of $(\Z_2)^m$. So there
         exists a group automorphism $\phi$ of $(\Z_2)^m$ so that
         $\phi(\lambda_1(P_i)) = \lambda_2(P_i)$ for each $1\leq i \leq k$.
         Define a homeomorphism $\Phi: V^n\times (\Z_2)^m \rightarrow V^n \times (\Z_2)^m$
         by
           $$\Phi(x,g)= (x,\phi(g)), \ \forall\,  x\in V^n \ \text{and}\ g\in (\Z_2)^m.
           $$
       Let $\theta_{\lambda_i}:  V^n\times (\Z_2)^m  \rightarrow M(V^n,\lambda_i)$ ($i=1,2$)
       be the quotient map defined in~\eqref{Glue_Back}.
       Then obviously $\theta_{\lambda_1}(x,g)=\theta_{\lambda_1}(x',g')$ if and only if
       $\theta_{\lambda_2}(\Phi(x,g))=\theta_{\lambda_2}(\Phi(x',g'))$.
       So $\Phi$ induces a homeomorphism $\widetilde{\Phi}$ from $M(V^n,\lambda_1)$ to
       $M(V^n,\lambda_2)$ by:
         $$ \widetilde{\Phi}(\theta_{\lambda_1}(x,g))= \theta_{\lambda_2}(\Phi(x,g)).$$
       Moreover, $\widetilde{\Phi}$ relates
       the natural $(\Z_2)^m$-actions on $M(V^n,\lambda_1)$ and $M(V^n,\lambda_2)$
       by:
        $$ \widetilde{\Phi}(g'\cdot \theta_{\lambda_1}(x,g))= \phi(g')\cdot
        \widetilde{\Phi}(\theta_{\lambda_1}(x,g)),\ \, \forall\, g' \in (\Z_2)^m.$$
        In addition, it is easy to see that
         $\widetilde{\Phi}$ induces the identity map on
        the orbit space $Q^n$. So by the definition, $M(V^n,\lambda_1)$ and
        $M(V^n,\lambda_2)$ are equivalent as principal $(\Z_2)^m$-bundles over
        $Q^n$.
       \end{proof}
      \vskip .2cm

       \begin{lem} \label{Lem:Max_Homeo_2}
        Suppose $M_1$ and $M_2$ are two principal $(\Z_2)^k$-bundles
        over $Q^n$, where $k=\dim_{\Z_2}H_{n-1}(Q^n,\Z_2)$. If $M_1$ and
        $M_2$ are both connected, $M_1$ must be equivalent to
        $M_2$.
      \end{lem}
      \begin{proof}
          By the notations in the above discussion and
          Theorem~\ref{Thm:Equiv-Isom}, we have
           $$M_i \cong M(V^n,\lambda_i)
          \ \text{for some}\ \lambda_i\in \mathrm{Col}_k(V^n), \ i=1,2 $$
          In addition, since $M_1$ and $M_2$ are both connected, Theorem~\ref{thm:comp} implies
          that $\mathrm{rank}(\lambda_1) = \mathrm{rank}(\lambda_2) = k$, i.e. $\lambda_1$
          and $\lambda_2$ are both maximally independent. So by
          Lemma~\ref{Lem:Max_Homeo}, $M(V^n,\lambda_1)$ and $M(V^n,\lambda_2)$ are
          equivalent as principal $(\Z_2)^k$-bundles over $Q^n$.
      \end{proof}
      \vskip .2cm

    Next, let us study some relations between
    $M(V^n, \lambda)$ for different $\lambda \in \mathrm{Col}_m(V^n)$.
    For the sake of conciseness, for any topological space $B$ and
    any field $\mathbb{F}$, we define
         $$ \quad \mathrm{hrk}(B, \mathbb{F}) : =
         \sum^{\infty}_{i=0} \dim_{\mathbb{F}} H^i(B, \mathbb{F}). $$

         \vskip .2cm

     \begin{lem} \label{Lem:Double_Covering}
           For any double covering $\xi: \widetilde{B} \rightarrow B$ and  $\forall\, i\geq 0$,
           $ \dim_{\Z_2} H^i(\widetilde{B},\Z_2) \leq  2 \cdot \dim_{\Z_2}
           H^i(B,\Z_2)$. So
          $ \mathrm{hrk}(\widetilde{B},\Z_2) \leq 2 \cdot
          \mathrm{hrk}(B,\Z_2)$.
       \end{lem}
       \begin{proof}
          The Gysin sequence of $\xi: \widetilde{B} \rightarrow B$ in $\Z_2$-coefficient reads:
      \[     \cdots \longrightarrow  H^{i-1}(B,\Z_2) \overset{\phi_{i-1}}{\longrightarrow} H^i(B,\Z_2)
            \overset{\xi^*}{\longrightarrow} H^i(\widetilde{B},\Z_2) \longrightarrow H^i(B,\Z_2)
            \overset{\phi_i}{\longrightarrow}  \cdots
                 \]
         where $e\in H^1(B)$ is the Euler class (or first Stiefel-Whitney
         class) of $\widetilde{B}$,
          and $\phi_i(\gamma) = \gamma \cup e,\ \forall\, \gamma \in H^i(B,\Z_2)$. Then
          by the exactness of the Gysin sequence,
       \begin{align*}
            \dim_{\Z_2} H^i(\widetilde{B},\Z_2) &= \dim_{\Z_2} H^i(B,\Z_2) -
         \dim_{\Z_2}\text{Im}(\phi_{i-1}) + \dim_{\Z_2}\text{ker}(\phi_i) \\
            &= 2 \cdot \dim_{\Z_2} H^i(B,\Z_2) -  \dim_{\Z_2}\text{Im}(\phi_{i-1}) -
            \dim_{\Z_2}\text{Im}(\phi_i) \\
            & \leq  2 \cdot \dim_{\Z_2} H^i(B,\Z_2)
      \end{align*}
       \end{proof}

       \begin{rem}
       In Lemma~\ref{Lem:Double_Covering},
       if we replace the $\Z_2$-coefficient by $\Z_p$ ($p$ is an odd prime) or rational
        coefficient, the conclusion in the lemma might fail in some cases.
       \end{rem}
  \vskip .2cm

     For any panel $P_j \subset \mathcal{P}(V^n)$, we define the following
        space which will play an important role later.
        \begin{equation} \label{Equ:Partial_Glue-back}
           M_{\backslash P_j}(V^n,\lambda) := V^n\times (\Z_2)^m \slash
           \sim_{ P_j}
           \qquad\qquad
           \end{equation}
        where  $(x,g) \sim_{P_j} (x',g')$ whenever
         $x' = \tau_i(x)$ for some $P_i \neq P_j$ and
              $g' = g+ \lambda(P_i) \in (\Z_2)^m$. In other words,
       $M_{\backslash P_j}(V^n,\lambda)$ is the quotient space
       of $V^n\times (\Z_2)^m$ under the rule in~\eqref{Glue_Back} except that
       we leave the interior of those facets in $P_j\times (\Z_2)^m$ open.
       We call $ M_{\backslash P_j}(V^n,\lambda)$ a \textit{partial glue-back}
       from $(V^n,\lambda)$. Let the corresponding quotient map be
    $ \theta^{\backslash P_j}_{\lambda} : V^n \times (\Z_2)^m \rightarrow
    M_{\backslash P_j}(V^n,\lambda) $. Then
    $\theta^{\backslash P_j}_{\lambda}(P_j\times (\Z_2)^m)$ is the boundary of
    $M_{\backslash P_j}(V^n,\lambda)$.
    \vskip .2cm

     \begin{lem} \label{Lem:Max_Indep}
        Suppose $\lambda_{max} \in \mathrm{Col}_k(V^n) $ is a
         maximally independent $(\Z_2)^k$-coloring on $V^n$,  where
         $k=\dim_{\Z_2}H_{n-1}(Q^n, \Z_2)$.
         Then for any $\lambda \in \mathrm{Col}_k(V^n)$,
           \[  \mathrm{hrk} (M(V^n,\lambda),\Z_2) \geq
           \mathrm{hrk} (M(V^n,\lambda_{max}),\Z_2) .   \]
      \end{lem}
      \begin{proof}
        Without loss of generality, suppose $\{ \lambda(P_1),\cdots, \lambda(P_s) \}$
        is a $\Z_2$-linear basis of $L_{\lambda}$. Choose
        $\omega_1, \cdots, \omega_{k-s} \in (\Z_2)^k$
        so that $\{ \lambda(P_1),\cdots, \lambda(P_s),  \omega_1, \cdots, \omega_{k-s}  \}$
        forms a $\Z_2$-linear basis of $(\Z_2)^k$. Then we define a sequence of
         $(\Z_2)^k$-colorings $\lambda_0,\cdots, \lambda_{k-s}$ on $V^n$
      as following. For any $0\leq j \leq k-s$, let
        \[  \lambda_j(P_i) :=  \left\{
          \begin{array}{ll}
             \lambda(P_i),  & \hbox{$1\leq i \leq s$ or $s+j<i\leq k$;} \\
             \omega_{i-s}, & \hbox{$s+1\leq i \leq s+j$.}
          \end{array}
        \right.
       \]
      Obviously, $\lambda_0=\lambda$, $L_{\lambda} = L_{\lambda_0} \subset  L_{\lambda_1}
     \subset \cdots \subset L_{\lambda_{k-s}} = (\Z_2)^k$ and
      $\dim_{\Z_2}(L_{\lambda_{j+1}}) = \dim_{\Z_2}(L_{\lambda_j}) +1$.
   So $\lambda_{k-s} \in \mathrm{Col}_k(V^n)$ is maximally independent.
   By Lemma~\ref{Lem:Max_Homeo}, $\mathrm{hrk} (M(V^n,\lambda_{max}),\Z_2) = \mathrm{hrk}
   (M(V^n,\lambda_{k-s}),\Z_2)$. Then it suffices to show that
    $\mathrm{hrk}(M(V^n,\lambda_{j-1}),\Z_2)  \geq  \mathrm{hrk}(M(V^n,\lambda_{j}),\Z_2)$
     for any $1\leq j \leq k-s$. \vskip .2cm

     Notice that the only difference between
    $\lambda_j$ and $\lambda_{j-1}$ is that:
    $\lambda_j(P_{s+j})= \omega_j$ while $\lambda_{j-1}(P_{s+j})=\lambda(P_{s+j})$. So
    $L_{\lambda_j} = L_{\lambda_{j-1}} \oplus \langle \omega_j \rangle \subset (\Z_2)^k$.
     Let
    $\theta_j: V^n\times (\Z_2)^k \rightarrow M(V^n,\lambda_j)$ be the quotient map
    defined by~\eqref{Glue_Back} for each $j$.\vskip .2cm

    For a fixed $j$, let $\widetilde{K}$ and $K$ be a connected component of
    $M(V^n,\lambda_j)$ and $M(V^n,\lambda_{j-1})$ respectively.
     By Theorem~\ref{thm:comp}, we can assume that:
      $$ \widetilde{K} = \theta_{j}(V^n\times L_{\lambda_{j}}),\quad
      K = \theta_{j-1}(V^n\times L_{\lambda_{j-1}}).$$

   Next, we define a free involution
    $\eta$ on $\widetilde{K}$ by: for any $[(x,g)] \in
    \widetilde{K}$,
     \begin{equation} \label{Equ:Involution}
        \eta([(x,g)])=\left( \lambda(P_{s+j})+ \omega_j\right)\cdot [(x,g)]
     \overset{\eqref{Equ:FreeAction}}{=} [(x,g+\lambda(P_{s+j})+ \omega_j)].
     \end{equation}
         \vskip .2cm
      \textbf{Claim:} the orbit space of $\widetilde{K}$ under the free involution
      $\eta$ is homeomorphic to $K$. So $\widetilde{K}$ is a double covering of $K$
        (see Example~\ref{Exam:Covering}).
      \vskip .2cm

      To prove the claim, first let
      $\theta^{\backslash P_{s+j}}_j : V^n\times (\Z_2)^m \rightarrow M_{\backslash P_{s+j}}
        (V^n,\lambda_j)$ be the quotient map of the partial glue-back defined
        by~\eqref{Equ:Partial_Glue-back}. For any $(x,g)\in V^n \times (\Z_2)^m$,
        denote
       $\overline{(x,g)} := \theta^{\backslash P_{s+j}}_j (x,g)$.
       And we define
        $$ W^n := \theta^{\backslash P_{s+j}}_j(V^n\times
        L_{\lambda_j}). $$
        Geometrically, $W^n$ is the quotient space
       of $V^n\times L_{\lambda_j}$ under $\theta_j$ except that we do not
      glue those facets in $P_{s+j} \times L_{\lambda_j}$.
      By the definition, $W^n = W^n_0 \cup W^n_1$ where
       $$ W^n_0 = \theta^{\backslash P_{s+j}}_j(V^n\times  L_{\lambda_{j-1}}),\quad
        W^n_1 = \theta^{\backslash P_{s+j}}_j (V^n\times
      ( L_{\lambda_{j-1}} + \omega_{j} )).$$
       Let  $ A_0 = \theta^{\backslash P_{s+j}}_j( P_{s+j}\times
      L_{\lambda_{j-1}}) \subset \partial W^n_0$,
      $ A_1 = \theta^{\backslash P_{s+j}}_j( P_{s+j}\times (L_{\lambda_{j-1}}+\omega_j))
       \subset \partial W^n_1$. \vskip .2cm

      Here, the fact that $\omega_j$ is linearly independent from
      $L_{\lambda_{j-1}}$ is essential for these constructions. Otherwise,
       $W^n_0$ and $W^n_1$ would be the same space.
         \vskip .2cm

       It is easy to see that $\widetilde{K}$ is the gluing of
       $W^n_0$ and $W^n_1$ by a homeomorphism
       $\varphi : A_0 \rightarrow A_1$ defined by:
       for $\forall \, x_0\in P_{s+j}$ and $\forall\, g_0\in L_{\lambda_{j-1}}$,
       $$ \overline{(x_0,g_0)} \in A_0 \overset{\varphi}{\longrightarrow}
       \overline{(\tau_{s+j}(x_0), g_0+ \omega_j)} \in A_1. $$
     Let $p: W^n =W^n_0 \cup W^n_1  \rightarrow W^n_0 \cup_{\varphi} W^n_1 =\widetilde{K}$
     denote this quotient map.
      So by our notations,
      $p\,(\,\overline{(x,g)}\,) = [(x,g)]$ for any $\overline{(x,g)} \in W^n$. \vskip .2cm

     Obviously, we have $\widetilde{K} = p(W^n_0) \cup p(W^n_1)$
     and $p(W^n_0) \cap p(W^n_1) = p(A_0) = p(A_1)$. The key observation here
     is that the involution $\eta$ maps
     $p(W^n_0)$ homeomorphically to $p(W^n_1)$, and the action of
     $\eta$ on $p(A_0) = p(A_1)$ is: for any $\overline{(x_0,g_0)}\in A_0$,
       \begin{align*}
         \eta ( p\, (\,\overline{(x_0,g_0)}\,) ) & = \eta([(x_0,g_0)]) = \eta([(\tau_{s+j}(x_0), g_0+ \omega_{j})]) \\
              & \overset{\eqref{Equ:Involution}}{=}
               [(\tau_{s+j}(x_0), g_0 + \lambda(P_{s+j}))] =
              p(\, \overline{(\tau_{s+j}(x_0), g_0+ \lambda(P_{s+j}))} \, )
         \end{align*}
        So the orbit space of $\widetilde{K}$ under the action of $\eta$ is homeomorphic to
        the quotient space of $W^n_0$ by identifying its boundary
        point $\overline{(x_0,g_0)}\in A_0 $
        with another point $\overline{(\tau_{s+j}(x_0), g_0+ \lambda(P_{s+j}))}\in A_0$, which is exactly the same as
        $\theta_{j-1}(V^n\times L_{\lambda_{j-1}}) = K$
        (see Example~\ref{Exam:Covering}). So our claim is proved.
        \vskip .2cm

    Then by Lemma~\ref{Lem:Double_Covering}, $\mathrm{hrk}(\widetilde{K},\Z_2) \leq
    2\cdot \mathrm{hrk}(K,\Z_2)$. Moreover, by
    Theorem~\ref{thm:comp}, the connected components in each
    $M(V^n,\lambda_{j})$ are pairwise homeomorphic and
    the number of connected components
    of $M(V^n,\lambda_{j-1})$ is twice that of $M(V^n,\lambda_{j})$,
    so we have $ \mathrm{hrk}(M(V^n,\lambda_{j-1}),\Z_2)\geq
   \mathrm{hrk}(M(V^n,\lambda_{j}),\Z_2)$. The lemma is proved.
  \end{proof}
  \vskip .7cm

  \section{Proof of Theorem~\ref{thm:main}}  \label{Sec3}

   First, we quote a lemma shown in~\cite{Uto09}.
   But we will slightly rephrase the original statement of this lemma
    to adapt to our proof of Theorem~\ref{thm:main}.\vskip .2cm

   \begin{lem}[Ustinovsky~\cite{Uto09}] \label{Lem:Connect}
      Let $(X,A)$ be a pair of CW-complexes such that $A$ has a
      collar neighborhood $U(A)$ in $X$, that is, $(U(A),A) \cong (A\times [0,1),A\times
      0)$. Suppose we have a homeomorphism $\varphi:  A \rightarrow A$ which can be extended to
      a homeomorphism $\widetilde{\varphi} : X \rightarrow X$.
      Let $Y=X_1 \cup_{\varphi} X_2$ be the space obtained by gluing two copies of $X$
      along $A$ via the map $\varphi$.
      Then for any field $\mathbb{F}$, we have:
        $\mathrm{hrk}(Y,\mathbb{F})\geq \mathrm{hrk}(A,\mathbb{F})$.
   \end{lem}
   \begin{proof}
      The argument here is almost the same as in~\cite{Uto09}.
      Let $U_1(A)$ and $U_2(A)$ be
       the collar neighborhoods of $A$ in $X_1$ and $X_2$
       respectively.
     Consider an open cover
      $Y=W_1 \cup W_2$ where $W_1 = X_1\cup U_2(A)$ and $W_2=X_2 \cup
      U_1(A)$. Then the Mayer-Vietoris sequence of cohomology groups for this open cover
      reads (we omit all the coefficient $\mathbb{F}$ below):
      \[ \cdots \rightarrow H^{j-1}(W_1 \cap W_2) \overset{\delta^*_{(j)}}{\longrightarrow}
       H^j(Y) \overset{g^*_{(j)}}{\longrightarrow} H^j(W_1)\oplus H^j(W_2)
       \overset{p^*_{(j)}}{\longrightarrow} H^j(W_1\cap W_2) \rightarrow \cdots \]

       Here the map $p^*_{(j)} = i_1^*\oplus -i_2^*$, where
       $i_1$ and $i_2$ are inclusions of $W_1\cap W_2$ into $W_1$
       and $W_2$ respectively. Since $W_1$ and $W_2$ are both homotopy equivalent to
      $X$ and $W_1\cap W_2 = U_1(A) \cup U_2(A) \cong A \times
      (-1,1)$, we get another long exact sequence which is equivalent to the above one:
     \[ \cdots \longrightarrow H^{j-1}(A)
     \overset{\widehat{\delta}^*_{(j)}}{\longrightarrow} H^j(Y)
     \overset{\widehat{g}^*_{(j)}}{\longrightarrow}
      H^j(X_1)\oplus H^j(X_2)
      \overset{\widehat{p}^*_{(j)}}{\longrightarrow} H^j(A) \longrightarrow \cdots \]
      Notice that the $\widehat{p}^*_{(j)} = \iota_1^*\oplus -(\iota_2\circ \varphi)^*$
      where
       $\iota_1$ and $\iota_2$ are inclusions of $A$ into $X_1$
       and $X_2$ respectively. For any $\gamma \in H^j(X_1)$,
       it is easy to see that
       $(\gamma, (\widetilde{\varphi}^{-1})^*\gamma)$ is in
       $\mathrm{ker}(\widehat{p}^*_{(j)})$. This implies that
       $\dim\mathrm{ker}(\widehat{p}^*_{(j)}) \geq \dim H^j(X)$ and
       so $\dim\mathrm{Im}(\widehat{p}^*_{(j)}) \leq \dim H^j(X)$ . Then we have:
       \begin{align*}
         \dim H^j(Y) &= \dim \mathrm{ker}(\widehat{g}^*_{(j)}) + \dim \mathrm{Im}(\widehat{g}^*_{(j)})
                    = \dim \mathrm{Im}(\widehat{\delta}^*_{(j)}) + \dim
                    \mathrm{ker}(\widehat{p}^*_{(j)})
                                   \\
                 &\geq \dim H^{j-1}(A) - \dim
                  \mathrm{Im}(\widehat{p}^*_{(j-1)}) + \dim H^j(X)  \\
                &\geq \dim H^{j-1}(A) - \dim H^{j-1}(X) + \dim
                H^j(X).
     \end{align*}
      By summing up these inequalities over all indices $j$, we
      get:
      \begin{align*}
        \mathrm{hrk}(Y,\F) = \sum_{j} \dim H^j(Y) &\geq \sum_{j} \dim H^{j-1}(A) - \dim H^{j-1}(X) + \dim
                H^j(X) \\
              &\quad = \sum_{j} \dim H^{j-1}(A) =
              \mathrm{hrk}(A,\F).
        \end{align*}
   \end{proof}

 \begin{rem}
   In the above lemma, the assumption that $\varphi: A\rightarrow A$
  can be extended to a homeomorphism $\widetilde{\varphi}: X \rightarrow X$
  is essential, otherwise the claim may not
  be true.
  \end{rem}
   \vskip .3cm

   \textbf{Proof of Theorem~\ref{thm:main}:}
     We shall organize the proof by an induction on the dimension of $M^n$.
   When $n=1$, since a principal $(\Z_2)^m$-bundle over a circle must be
   a disjoint union of $2^m$ or $2^{m-1}$ circles,
   so the theorem holds. Then we assume
   the theorem holds for manifolds with dimension less than $n$.
   \vskip .2cm

    Suppose $P^n$ is an $n$-dimensional simple convex polytope with $k+n$ facets
    $F_1,\cdots, F_{k+n}$ ($k\geq 1$) and
    $\pi_{\mu}: Q^n \rightarrow P^n$ is a small cover over $P^n$ with the
    characteristic function $\mu$. For any face
    $\mathbf{f}=F_{i_1}\cap \cdots \cap F_{i_l}$ of $P^n$, let
    $G^{\mu}_{\mathbf{f}}$ be the rank-$l$ subgroup of $(\Z_2)^n$ generated by
    $\mu(F_1),\cdots, \mu(F_l)$.
     Then by the definition,
     \begin{equation} \label{Equ:Glue-SmallCov}
       \text{$Q^n= P^n\times (\Z_2)^n \slash
    \sim$,\ \ $(p,w)\sim (p',w') \Longleftrightarrow p=p', w-w'\in
    G^{\mu}_{\mathbf{f}(p)}$},
    \end{equation}
     where $\mathbf{f}(p)$ is the unique face of $P^n$ that contains
    $p$ in its relative interior.  It was shown in~\cite{DaJan91}
    that the $\Z_2$-Betti numbers of $Q^n$ can be computed from the
   $h$-vector of $P^n$. In particular, $H_{n-1}(Q^n,\Z_2) \cong
   (\Z_2)^{k}$. \vskip .2cm

     Next, we choose an
     arbitrary vertex $v_0$ of $P^n$. By re-indexing the facets of $P^n$,
     we can assume $F_1,\cdots, F_k$ are those facets of $P^n$ that are not
     incident to $v_0$. Then according to~\cite{DaJan91},
      the homology classes of the facial submanifolds
      $\pi_{\mu}^{-1}(F_1),\cdots, \pi_{\mu}^{-1}(F_k)$ form a $\Z_2$-linear basis of
      $H_{n-1}(Q^n,\Z_2)$.
    Cutting $Q^n$ open along $\pi_{\mu}^{-1}(F_1),\cdots,
     \pi_{\mu}^{-1}(F_k)$ will give us
    a $\Z_2$-core of $Q^n$, denoted by $V^n$.
     We can think of $V^n$
    as a partial gluing of the $2^n$ copies of $P^n$ according to the rule
    in~\eqref{Equ:Glue-SmallCov} except that we leave the facets
    $F_1,\cdots, F_k$ in each copy of $P^n$ open (see Figure~\ref{p:Panel_biject_3} for example).
     Let
    $\zeta : P^n\times (\Z_2)^n \rightarrow V^n$ denote the quotient
    map and let
    $P_1,\cdots, P_k$ be the panels of $V^n$ corresponding to
    $\pi_{\mu}^{-1}(F_1),\cdots,\pi_{\mu}^{-1}(F_k)$. Then
    each $P_i$ consists of $2^n$ copies of $F_i$ and the involutive panel structure
     $\{\tau_i: P_i \rightarrow P_i\}_{1\leq i \leq k}$ on $V^n$ can be written as:
     \begin{equation} \label{Equ:Invol-Panel}
      \tau_i(\zeta(p,w)) = \zeta(p, w+\mu(F_i) ), \ \,
      \forall\, p \in F_i, \forall w\in (\Z_2)^n
     \end{equation}
     Obviously, each $\tau_i$ extends to an automorphism $\widetilde{\tau}_i$ of
      $V^n$ given by the same form:
      \begin{equation} \label{Equ:Commute}
       \widetilde{\tau}_i(\zeta(p,w)) = \zeta(p, w+\mu(F_i)), \ \,
      \forall\, p \in P^n, \forall w\in (\Z_2)^n
      \end{equation}
      And these $\widetilde{\tau}_i$ commute with each other, i.e.
      $\widetilde{\tau}_i \circ \widetilde{\tau}_j =
      \widetilde{\tau}_j \circ \widetilde{\tau}_i$, $1\leq i , j \leq k$.
       So
      each $\widetilde{\tau}_i$ will preserve any panel $P_j$ of $V^n$.\vskip .2cm

      \begin{figure}
         \includegraphics[width=0.45\textwidth]{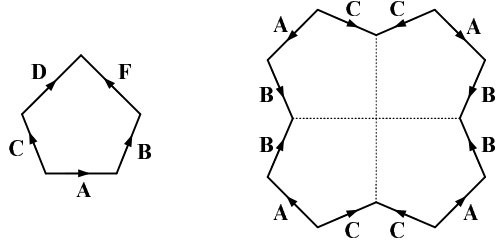}
          \caption{A $\Z_2$-core of a small cover in dimension $2$}\label{p:Panel_biject_3}
      \end{figure}

      To prove Theorem~\ref{thm:main}, it suffices to show that
        $\mathrm{hrk}(M(V^n,\lambda),\Z_2) \geq 2^m$ for any $\lambda
        \in\mathrm{Col}_m(V^n)$ (because of Theorem~\ref{Thm:Equiv-Isom}). \vskip .2cm

      First, we assume $m=k$.
      Let $\lambda_0$ be a maximally independent $(\Z_2)^k$-coloring of
      $V^n$, i.e. $\mathrm{rank}(\lambda_0)=k$.
      By Lemma~\ref{Lem:Max_Indep}, for
     $\forall\, \lambda \in \mathrm{Col}_k(V^n)$,
      $\mathrm{hrk}(M(V^n,\lambda),\Z_2) \geq
      \mathrm{hrk}(M(V^n,\lambda_0),\Z_2)$. So it suffices to prove
      that
        \begin{equation} \label{Equ:Inequal}
          \mathrm{hrk}(M(V^n,\lambda_0),\Z_2) \geq 2^k.
        \end{equation}

       Indeed, the~\eqref{Equ:Inequal} follows from Theorem~\ref{thm:CaoLU} and
      Lemma~\ref{Lem:Max_Homeo_2} (see the Remark~\ref{rem:Same} below).
      But here we will give another proof of~\eqref{Equ:Inequal}
      which only uses the Lemma~\ref{Lem:Connect} taken from~\cite{Uto09}.
      Our proof will take advantage of the special symmetries of small covers
      (see~\eqref{Equ:Invol-Panel} and~\eqref{Equ:Commute}), and
      it is more natural from the viewpoint of the
      glue-back construction.
      \vskip .2cm

    Since $\lambda_0$ is maximally independent, by Lemma~\ref{Lem:Max_Homeo}, we can assume
   $\lambda_0(P_i)=e_i$, $1\leq i \leq k$,
   where $\{ e_1,\cdots, e_k \}$ is a linear basis of $(\Z_2)^k$.
   Let $\theta_{\lambda_0} : V^n \times (\Z_2)^k \rightarrow M(V^n,\lambda_0)$ be the
    quotient map defined by~\eqref{Glue_Back}.
    \vskip .2cm

   Now take an arbitrary panel of $V^n$, say $P_1$ and
   let $M_{\backslash P_1}(V^n,\lambda_0)$ be a partial glue-back
   from $(V^n,\lambda_0)$ defined by~\eqref{Equ:Partial_Glue-back}.
   Let
    $\theta^{\backslash P_1}_{\lambda_0} : V^n \times (\Z_2)^k \rightarrow
    M_{\backslash P_1}(V^n,\lambda_0) $ be the corresponding quotient map.
     Suppose $H$ is the subgroup of $(\Z_2)^k$ generated by
    $\{e_2,\cdots, e_k\}$. Then we define:
    \begin{equation} \label{Equ:Y}
     Y_1 =    \theta^{\backslash P_1}_{\lambda_0} (V^n \times H), \ \; Y_2 =
    \theta^{\backslash P_1}_{\lambda_0} (V^n \times (e_1+H));
    \end{equation}
    \begin{equation}
       A_1 =    \theta^{\backslash P_1}_{\lambda_0} (P_1 \times H), \ \; A_2 =
    \theta^{\backslash P_1}_{\lambda_0} (P_1\times (e_1+H)).
    \end{equation}
   \vskip .2cm

    Obviously, $A_1=\partial Y_1$, $A_2=\partial Y_2$ and
     there is homeomorphism $\Pi: Y_1 \rightarrow Y_2$ with
     $\Pi(A_1) = A_2$. Indeed, $\Pi$ is given by:
     $$ \Pi(\theta^{\backslash P_1}_{\lambda_0}(x,h)) =
      \theta^{\backslash P_1}_{\lambda_0}(x,h + e_1),\ \forall\, x\in V^n,\, \forall\, h\in H.$$

    It is easy to see that $M(V^n,\lambda_0)$ is the gluing of
    $Y_1$ and $Y_2$ along their boundary
    by a homeomorphism $\varphi : A_1 \rightarrow A_2$ defined by:
     \[ \varphi(\theta^{\backslash P_1}_{\lambda_0}(x_1,h)) = \theta^{\backslash P_1}_{\lambda_0}
       (\tau_1(x_1),h + e_1), \ \forall\, x_1\in P_1,\, \forall\, h\in H.  \]

      Moreover, since $\tau_1: P_1 \rightarrow P_1$
      extends to a homeomorphism $\widetilde{\tau}_1 : V^n \rightarrow V^n$
      (see~\eqref{Equ:Invol-Panel} and~\eqref{Equ:Commute}),
       we can extend $\varphi$ to a homeomorphism $\widetilde{\varphi}: Y_1 \rightarrow Y_2$ by:
       $$ \widetilde{\varphi}(\theta^{\backslash P_1}_{\lambda_0}(x,h)) =
      \theta^{\backslash P_1}_{\lambda_0}(\widetilde{\tau}_1(x),h + e_1),\ \forall\, x\in V^n,\, \forall\, h\in H.$$
       The $\widetilde{\varphi}$ is well-defined since $\widetilde{\tau}_1$
       commutes with each $\tau_i$ on $P_i$ (see~\eqref{Glue_Back} and~\eqref{Equ:Commute}).
        \vskip .2cm

     Now, if we identify $(Y_1,A_1)$ with $(Y_2, A_2)$ via $\Pi$,
    we get a decomposition of $M(V^n,\lambda_0)$ that satisfies all the conditions in
     Lemma~\ref{Lem:Connect}.
     So Lemma~\ref{Lem:Connect} implies:
       \begin{equation} \label{Equ:Estimate-1}
     \mathrm{hrk}(M(V^n,\lambda_0),\Z_2) \geq
     \mathrm{hrk}(A_1,\Z_2).
     \end{equation}

     In addition, let $ q : Y_1\cup Y_2 \rightarrow M(V^n,\lambda_0)$ be the quotient
       map.  It is easy to see that:
  $A_1 \cong q (A_1)= \theta_{\lambda_0}^{-1} (\Sigma_1)$. Since
  $\theta_{\lambda_0}^{-1} (\Sigma_1)$ is a principal
  $(\Z_2)^k$-bundle over $\Sigma_1$ and $\Sigma_1$ is a small cover over $F_1$ with dimension
  $n-1$, so
     by the induction hypothesis, we have
     $\mathrm{hrk}(\theta_{\lambda_0}^{-1} (\Sigma_1),\Z_2) \geq 2^k$. So
      $ \mathrm{hrk}(A_1,\Z_2) =
      \mathrm{hrk}(\theta_{\lambda_0}^{-1} (\Sigma_1),\Z_2) \geq
      2^k$.
     Then by~\eqref{Equ:Estimate-1}, we get
     $\mathrm{hrk}(M(V^n,\lambda_0),\Z_2) \geq 2^k$. So this case is confirmed.
      \vskip .2cm

     Next, we assume $m<k$.  Let
      $\iota : (\Z_2)^m \hookrightarrow (\Z_2)^k$ be the standard
      inclusion and define $\widehat{\lambda}:= \iota \circ
      \lambda$. We consider $\widehat{\lambda}$ as a $(\Z_2)^k$-coloring on $V^n$.
       So by the above argument, $\mathrm{hrk}(M(V^n,\widehat{\lambda}),\Z_2) \geq 2^k$.
      Since by Theorem~\ref{thm:comp}, $M(V^n,\widehat{\lambda})$
       consists of $2^{k-m}$ copies of $M(V^n,\lambda)$, so
       $\mathrm{hrk}(M(V^n,\lambda),\Z_2) \geq 2^m$.\vskip .2cm

       Finally, we assume $m>k$. Since $\mathrm{rank}(\lambda) \leq
       k$, with a proper change of basis, we can assume $L_{\lambda} \subset (\Z_2)^k \subset
       (\Z_2)^m$. Let $\varrho: (\Z_2)^m \rightarrow (\Z_2)^k$ be
       the standard projection. Define $\overline{\lambda}:= \varrho\circ \lambda$.
       Similarly, we consider $\overline{\lambda}$ as a $(\Z_2)^k$-coloring on
     $V^n$ and so we have $\mathrm{hrk}(M(V^n,\overline{\lambda}),\Z_2) \geq
     2^k$.  Since by Theorem~\ref{thm:comp}, $M(V^n,\lambda)$
       consists of $2^{m-k}$ copies of $M(V^n,\overline{\lambda})$, so
       $\mathrm{hrk}(M(V^n,\lambda),\Z_2) \geq 2^m$.
       \vskip .2cm
     So for $\forall\, m \geq 1 $ and $\forall\, \lambda \in
    \mathrm{Col}_m(V^n)$, we always have $\mathrm{hrk}(M(V^n,\lambda),\Z_2) \geq 2^m$.
    The induction is completed.  \hfill $\square$
    \vskip .2cm

       \begin{rem}\label{rem:Same}
        Notice that $M(V^n,\lambda_0)$ is a connected principal
      $(\Z_2)^k$-bundle over $Q^n$, so is the real moment-angle manifold
      $\mathcal{Z}_{P^n}$.
       Then by Lemma~\ref{Lem:Max_Homeo_2}, $M(V^n,\lambda_0)$ is
       homeomorphic to  $\mathcal{Z}_{P^n}$.
       Then the result of the Theorem~\ref{thm:CaoLU} also tells us
        that $\mathrm{hrk}(M(V^n,\lambda_0),\Z_2)\geq 2^k$.
       \end{rem}
  \vskip .2cm

         A crucial observation in the above proof is that: when
         $\lambda_0\in \mathrm{Col}_k(V^n)$ is
      maximally independent, we can always get the type of decomposition
      of $M(V^n,\lambda_0)$ as in Lemma~\ref{Lem:Connect}, which allows us to
       use the induction hypothesis. However, for an arbitrary $\lambda\in \mathrm{Col}_k(V^n)$,
       this type of decomposition for $M(V^n,\lambda)$ may not exist (at least not
       very obvious).\vskip .2cm

         For example,
        in the lower picture in Figure~\ref{p:Torus-Cover}, we have a
        principal $(\Z_2)^2$-bundle $\pi: M^2 \rightarrow T^2$ where $M^2$ is a disjoint union of
        two tori. The union of the two meridians in $M^2$ is the inverse image of
        a meridian in $T^2$ under $\pi$. If we cut $M^2$ open along these two meridians,
        we will get two circular cylinders.
         But $M^2$ here is not got by gluing these two cylinders
         together.
       This is because that the colors of the $(\Z_2)^2$-coloring on the two
        panels are not linearly independent. So the construction
        in~\eqref{Equ:Y} for this case fails to give us the type of
        decomposition of $M^2$ as in Lemma~\ref{Lem:Connect}.
        \vskip .2cm

         So when $\lambda\in \mathrm{Col}_k(V^n)$ is not maximally independent,
          we may not be able to
        directly apply the induction hypothesis to $M(V^n,\lambda)$ as we do
        to $M(V^n,\lambda_0)$ above. But these cases are settled by Lemma~\ref{Lem:Max_Indep}.

\section{Real Moment-Angle Manifolds from the Viewpoint of Glue-back construction}
  \label{Sec4}

   Suppose $P^n$ is an $n$-dimensional simple polytope with $k+n$
      facets $F_1,\cdots , F_{k+n}$ and $\pi_{\mu}: Q^n \rightarrow P^n$ is a small cover
       with a characteristic function $\mu$ on $P^n$. We know that
       $\mu : \{ F_1, \cdots, F_{k+n} \} \rightarrow (\Z_2)^n$
       satisfies: whenever $F_{i_1}\cap \cdots \cap F_{i_s} \neq \varnothing$,
        $\lambda(F_{i_1}), \cdots, \lambda(F_{i_s})$ are
        linearly independent vectors in $(\Z_2)^{n}$. For the
        convenience of our following discussion,
        let a linear basis of $(\Z_2)^n$ be $\{ e_{k+1},\cdots, e_{k+n}
        \}$.\vskip .2cm

       In this section, we will use the
       $\Z_2$-core $V^n$ of $Q^n$ as described at the beginning of the proof of
    Theorem~ \ref{thm:main}, whose involutive panel structure
    $\{ \tau_i: P_i \rightarrow P_i \}_{1\leq i \leq k}$ is defined by~\eqref{Equ:Invol-Panel}.
   \vskip .2cm

   It is well known that the real moment-angle manifold $\mathcal{Z}_{P^n}$
    is a principal $(\Z_2)^k$-bundle over
   $Q^n$. Since $\mathcal{Z}_{P^n}$ is connected, by Lemma~\ref{Lem:Max_Homeo_2},
    $\mathcal{Z}_{P^n}$ is homeomorphic to $M(V^n,\lambda_0)$
   for any maximally independent $\lambda_0 \in \mathrm{Col}_k(V^n)$.
   Now, let us compare the definitions of $\mathcal{Z}_{P^n}$ and
   $M(V^n,\lambda_0)$ and then construct a homeomorphism
   from $M(V^n,\lambda_0)$ to $\mathcal{Z}_{P^n}$ explicitly.
   \vskip .2cm

   Suppose $\{ e_1, \cdots, e_k \}$ is a linear basis of $(\Z_2)^k$.
    We choose $\lambda_0(P_{i}) = e_{i}$ for
    $1\leq i \leq k$. Let $\theta_{\lambda_0} : V^n \times (\Z_2)^k \rightarrow M(V^n,\lambda_0)$
    be the quotient map defined by~\eqref{Glue_Back}. By the definition,
   the panel $P_i$ consists of $2^n$ copies of $F_i$, and any $(x,g)\in P_i\times (\Z_2)^k$
    is identified with $(\tau_i(x), g+\lambda_0(P_{i})) = (\tau_i(x),
    g+e_i)$ under $\theta_{\lambda_0}$.\vskip .2cm

    Suppose $(\Z_2)^{k+n} = (\Z_2)^k \oplus (\Z_2)^n$ with a linear basis
    $\{ e_1, \cdots, e_k, e_{k+1}, \cdots, e_{k+n} \}$, and we identify
    $(\Z_2)^k$ and $(\Z_2)^n$ as a subgroup of $(\Z_2)^{k+n}$ in the obvious way. The
    real moment-angle manifold
    $\mathcal{Z}_{P^n}$ corresponds to a $(\Z_2)^{k+n}$-coloring $\mu_0$
   on $P^n$ which is $\mu_0(F_i) = e_i$ for any $1\leq i \leq k+n$. By the definition,
   $\mathcal{Z}_{P^n}$ is obtained by gluing $2^{k+n}$ copies of $P^n$ together
   by identifying any $(p,\widetilde{g}) \in F_i \times (\Z_2)^{k+n}$ with $(p,\widetilde{g}+\mu_0(F_i))$
   for all facet $F_i$. Let $\Theta: P^n \times (\Z_2)^{k+n} \rightarrow \mathcal{Z}_{P^n}$
   be the corresponding quotient map. \vskip .2cm

   To see the relationship between $\mathcal{Z}_{P^n}$ and
   $M(V^n,\lambda)$, let us decompose the above gluing process defined by
   $\Theta$ into two steps.
   In the first step, we glue the $2^{k+n}$ copies of
   $P^n$ only along the facets $F_{k+1}, \cdots, F_{k+n}$ on their
   boundaries. Then we will get $2^k$ copies of $V^n$, each of which is the gluing of
   $2^n$ copies of $P^n$.
    We readily index these $V^n$'s by the elements of $(\Z_2)^k =
     \langle e_1, \cdots, e_k \rangle \subset (\Z_2)^{k+n}$.
    Let $\widetilde{\zeta} : P^n \times (\Z_2)^{k+n} \rightarrow V^n \times (\Z_2)^k$
     denote this partial gluing map,
     and let $J = (\Z_2)^n = \langle e_{k+1}, \cdots, e_{k+n} \rangle \subset (\Z_2)^{k+n}$.
     Then we have:
     \begin{equation} \label{Equ:Vn-copy}
       \  V^n\times g = \widetilde{\zeta} (P^n \times (g+J)), \
        \forall\, g\in (\Z_2)^k.
      \end{equation}

    The $(\Z_2)^{k+n}$-coloring $\mu_0$ on $P^n$ induces a
   coloring $\widehat{\lambda}_0$ on $V^n$ valued in $(\Z_2)^k \subset (\Z_2)^{k+n}$ by:
    $\widehat{\lambda}_0(P_i) = \mu_0(F_i) = e_i,\ 1\leq i \leq k$.
     Note that the two $(\Z_2)^k$-coloring
      $\widehat{\lambda}_0$ and $\lambda_0$ on $V^n$ actually coincide, but they
      are used for different purposes.\vskip .2cm

     In the second step, the $\mathcal{Z}_{P^n}$ is obtained
   from gluing the $2^k$ copies of $V^n$ by identifying
   any $(x,g)\in P_i\times (\Z_2)^k$ with $(x, g+\widehat{\lambda}_0(P_{i})) = (x,
   g+e_i)$ for all $P_i$, $1\leq i \leq k$.
    Let $\vartheta_{\widehat{\lambda}_0} : V^n \times (\Z_2)^k \rightarrow
     \mathcal{Z}_{P^n}$ denote this quotient map. Obviously,
     $$\Theta = \vartheta_{\widehat{\lambda}_0} \circ \widetilde{\zeta} :
     P^n \times (\Z_2)^{k+n} \rightarrow \mathcal{Z}_{P^n}. $$

   Notice that the domains of $\theta_{\lambda_0}$ and $\vartheta_{\widehat{\lambda}_0}$
   are both $V^n\times (\Z_2)^k$.
    By comparing their definitions, we
   see that the difference between $\theta_{\lambda_0}$ and $\vartheta_{\widehat{\lambda}_0}$
   is just the involution
   $\tau_i$ on each panel $P_i$. Since each $\tau_i$ extends to an involution $\widetilde{\tau}_i$
   on $V^n$ (see~\eqref{Equ:Commute}), for
    any $g= \overset{k}{\underset{i=1}{\sum}}\, \varepsilon_i e_i \in
   (\Z_2)^k$ where $\varepsilon_i \in \{ 0, 1 \}$, we get an involution $\psi_g : V^n \rightarrow V^n$
   by:
    \begin{equation} \label{Equ:psi-g}
     \psi_g(x) := \widetilde{\tau}^{\varepsilon_k}_k \circ \cdots \circ
     \widetilde{\tau}^{\varepsilon_1}_1 (x),\ \forall\, x \in V^n.
     \end{equation}
    The $\psi_g$ is independent of the ordering of $ \widetilde{\tau}_1, \cdots,
    \widetilde{\tau}_k$ since they commute with each other.
     Using these involutions $\{ \psi_g \}_{g \in (\Z_2)^k}$, we can define a
     homeomorphism $\Psi : V^n \times (\Z_2)^k \rightarrow V^n \times (\Z_2)^k$ by:
    \[
     \Psi(x, g) := ( \psi_g(x), g ), \ \, \forall\, x\in V^n, \  \forall\, g\in
     (\Z_2)^k.
     \]
   Obviously, $\Psi\circ \Psi = id$. Moreover, we can show the following lemma.
   \vskip .2cm

   \begin{lem}
      $\theta_{\lambda_0}(x,g) =\theta_{\lambda_0}(x',g')$
   if and only if $\vartheta_{\widehat{\lambda}_0}(\Psi(x,g)) =
   \vartheta_{\widehat{\lambda}_0}(\Psi(x',g'))$.
   \end{lem}
      \begin{proof}
       If $\theta_{\lambda_0}(x,g) =\theta_{\lambda_0}(x',g')$,
       without loss of generality, we can assume that
        $$ \text{ $x\in P_i$ and $(x',g') = (\tau_i(x),g+\lambda_0(P_i)) = (\tau_i(x), g + e_i)$ for some $i$} .$$
       Then $\Psi(x',g') = (\psi_{g+e_i}(\tau_i(x)),g+e_i) =
        (\psi_g \circ \widetilde{\tau}_i (\tau_i(x)), g+e_i)=
         (\psi_g(x), g+e_i) $. So $\vartheta_{\widehat{\lambda}_0}(\Psi(x',g'))=
         \vartheta_{\widehat{\lambda}_0}(\psi_g(x), g+e_i)=
         \vartheta_{\widehat{\lambda}_0}(\psi_g(x), g)=
         \vartheta_{\widehat{\lambda}_0}(\Psi(x,g))$. \vskip .2cm

           Conversely, if $\vartheta_{\widehat{\lambda}_0}(\Psi(x,g)) =
   \vartheta_{\widehat{\lambda}_0}(\Psi(x',g'))$, without loss of
   generality, we can assume that $x\in P_i$ and $\Psi(x',g') = (\psi_{g'}(x'), g') = (\psi_g(x), g +
   e_i)$ for some $i$. Then we have $\psi_{g'}(x') = \psi_{g+e_i} (x') = \psi_g (\widetilde{\tau}_i (x')) =
   \psi_g(x)$. So $\widetilde{\tau}_i (x')= x $ and so $x'= \widetilde{\tau}_i(x) =
   \tau_i(x)$. Therefore, $\theta_{\lambda_0}(x',g')=
   \theta_{\lambda_0}(\tau_i(x),g+e_i)= \theta_{\lambda_0}(x,g)$.
    \end{proof}

   By the above lemma, $\Psi$ induces a homeomorphism $\widetilde{\Psi} : M(V^n,\lambda_0) \rightarrow
   \mathcal{Z}_{P^n}$ by:
     \begin{equation} \label{Equ:Glue_back-Cano_Action}
       \widetilde{\Psi}(\theta_{\lambda_0}(x,g)) :=
   \vartheta_{\widehat{\lambda}_0}(\Psi(x,g)), \ \, \forall\, x\in V^n, \  \forall\, g\in
     (\Z_2)^k.
     \end{equation}
   It is easy to see that $\widetilde{\Psi}^{-1}(\vartheta_{\widehat{\lambda}_0}(x,g)) =
   \theta_{\lambda_0}(\Psi(x,g)) =  \theta_{\lambda_0}(\psi_g(x), g
   )$.\vskip .2cm

   Next, let us see how $\widetilde{\Psi}$ relates the
   natural $(\Z_2)^k$-action on
   $ M(V^n,\lambda_0)$ and the canonical $(\Z_2)^{k+n}$-action on
   $\mathcal{Z}_{P^n}$. The natural action of $(\Z_2)^k$ on
   $ M(V^n,\lambda_0)$ defined by~\eqref{Equ:FreeAction} is
   \begin{equation} \label{Equ:FreeAction-3}
    g' \cdot \theta_{\lambda_0}(x,g) = \theta_{\lambda_0}(x,g'+g),
   \ \, \forall\, x\in V^n, \  \forall\, g', g\in (\Z_2)^k.
   \end{equation}
   This induces a free action of $(\Z_2)^k$ on $\mathcal{Z}_{P^n}$ through the homeomorphism
   $\widetilde{\Psi}$ by:
    \begin{equation} \label{Equ:FreeAction-4}
     g' \star \vartheta_{\widehat{\lambda}_0}(x,g) =
     \vartheta_{\widehat{\lambda}_0}(\psi_{g'}(x),g'+g),
  \ \, \forall\, x\in V^n,  \  \forall\, g', g\in (\Z_2)^k,
    \end{equation}
   so that $\widetilde{\Psi}$ is equivariant with respect to the
   $(\Z_2)^k$-actions defined by~\eqref{Equ:FreeAction-3}
   and~\eqref{Equ:FreeAction-4}:
    \begin{equation} \label{Equ:Equivar-Action-1}
     \widetilde{\Psi}(g'\cdot \theta_{\lambda_0}(x,g)) =
    g'\star \widetilde{\Psi}(\theta_{\lambda_0}(x,g)), \quad \forall\, g'\in
   (\Z_2)^k.
   \end{equation}

   On the other hand, the canonical $(\Z_2)^{k+n}$-action on
   $\mathcal{Z}_{P^n}$ is defined by:
   \begin{equation} \label{Equ:Cano-Action}
      \widetilde{g}_0 \circledast \Theta (p,\widetilde{g}) =  \Theta (p,  \widetilde{g}_0 +
      \widetilde{g}),\ \, \forall \, p\in P^n,\ \forall\,  \widetilde{g}_0 , \widetilde{g}
      \in  (\Z_2)^{k+n}.
   \end{equation}
   We can first interpret the free $(\Z_2)^k$-action $\star$ on $\mathcal{Z}_{P^n}$
   defined by~\eqref{Equ:FreeAction-4}
    as the restriction of the canonical $(\Z_2)^{k+n}$-action $\circledast$
    on $\mathcal{Z}_{P^n}$
    to a subtorus of $(\Z_2)^{k+n}$. In fact, by the
    definition of $\widetilde{\tau}_i$ in~\eqref{Equ:Commute}
    and~\eqref{Equ:Vn-copy}, if $g'= \overset{k}{\underset{i=1}{\sum}}\, \varepsilon'_i e_i \in
   (\Z_2)^k$ where $\varepsilon'_i \in \{ 0, 1 \}$ for $1\leq i \leq k$,
    the action of $g'$ on $\mathcal{Z}_{P^n}$ defined by~\eqref{Equ:FreeAction-4} is:
   \[
        g' \star \vartheta_{\widehat{\lambda}_0} \left(\widetilde{\zeta} \left(p,
        \widetilde{g} \right) \right)  =  \vartheta_{\widehat{\lambda}_0}
        \left(\widetilde{\zeta} \left(p,
         \widetilde{g} + g' + \sum^k_{i=1} \varepsilon'_i \mu(F_i) \right) \right),
         \ \, \forall\, (p,\widetilde{g})\in P^n\times (\Z_2)^{k+n}
     \]
  Since $ \vartheta_{\widehat{\lambda}_0} \circ \widetilde{\zeta} = \Theta$,
  so it is equivalent to write:
  \begin{equation} \label{Equ:Equivar-Action-2}
    \left( \sum^k_{i=1} \varepsilon'_i e_i \right)
     \star \Theta\left(p,\widetilde{g} \right)  =  \Theta \left(p,
         \widetilde{g} +  \sum^k_{i=1} \varepsilon'_i \left(e_i+\mu(F_i) \right) \right).
   \end{equation}

     Let $H_{\mu}$ be the subgroup of $(\Z_2)^{k+n}$ spanned by
     $\{ e_1 + \mu(F_1), \cdots, e_k + \mu(F_k) \}$. Since
     $\mu$ takes value in $(\Z_2)^n=\langle e_{k+1}, \cdots , e_{k+n}
     \rangle$, the rank of $H_{\mu}$ is equal to $k$.
     Let $\sigma: (\Z_2)^k \rightarrow H_{\mu}$ be a group
     isomorphism defined by
      \begin{equation} \label{Equ:Group-Isom}
        \sigma(e_i) = e_i + \mu(F_i), \quad  i=1,\cdots, k.
        \end{equation}
         Then according to~\eqref{Equ:Equivar-Action-1} ---~\eqref{Equ:Equivar-Action-2},
          we have:
   \begin{equation} \label{Equ:Equivar-Action-2-2}
     g'\star \Theta\left(p,\widetilde{g} \right) =
   \sigma(g') \circledast \Theta\left(p,\widetilde{g} \right),
    \ \, \forall\, (p,\widetilde{g})\in P^n\times (\Z_2)^{k+n}
   \end{equation}
     This implies that the free $(\Z_2)^k$-action on $\mathcal{Z}_{P^n}$ defined by~\eqref{Equ:FreeAction-4}
     is equivalent to the restriction of the canonical $(\Z_2)^{k+n}$-action
     on $\mathcal{Z}_{P^n}$ to $H_{\mu}$. By combining the~\eqref{Equ:Equivar-Action-1}
   and~\eqref{Equ:Equivar-Action-2-2}, we get:
   \begin{equation} \label{Equ:Equivar-Action-3}
        \widetilde{\Psi}(g'\cdot \theta_{\lambda_0}(x,g) ) =
   \sigma(g') \circledast \widetilde{\Psi}(\theta_{\lambda_0}(x,g)), \
   \; \forall\, \theta_{\lambda_0}(x,g) \in M(V^n,\lambda_0)
    \end{equation}

    So we have proved the following proposition.

  \begin{prop}
   The natural free $(\Z_2)^k$-action on $M(V^n,\lambda_0)$
     is equivalent to the restriction of the canonical $(\Z_2)^{k+n}$-action
     on $\mathcal{Z}_{P^n}$ to $H_{\mu}$.
  \end{prop}

  \begin{cor} \label{Cor:Equivalence}
    For any subgroup $N \subset (\Z_2)^k$,
     the natural $(\Z_2)^k\slash N$-action on $M(V^n,\lambda_0)\slash
     N$ is equivalent to the canonical $H_{\mu} \slash \sigma(N)$-action
     on $\mathcal{Z}_{P^n} \slash \sigma(N)$.
     \end{cor}\vskip .2cm

     In addition, it is easy to see that
     the intersection of $H_{\mu}$
       with the isotropy subgroup of any orbit of $\mathcal{Z}_{P^n}$
      under the canonical $(\Z_2)^{k+n}$-action is trivial. So the
      canonical action of $H_{\mu}$ on $\mathcal{Z}_{P^n}$ is indeed free.

       \vskip .2cm

      \begin{rem}
      The equivalence $\widetilde{\Psi}$
      identifies the orbit space $ M(V^n,\lambda_0) \slash (\Z_2)^k \cong Q^n$
      with
      the partial quotient $\mathcal{Z}_{P^n} \slash H_{\mu}$
      (see~\eqref{Equ:Equivar-Action-3}). Notice that if we choose
      another vertex $v'_0$ of $P^n$ and let $F'_1,\cdots, F'_k$ be the
      facets of $P^n$ that are not incident to $v'_0$, we will get
      another subtorus $H'_{\mu}$ of $(\Z_2)^{k+n}$ with rank $k$ by
      the above arguments
      so that $Q^n \cong \mathcal{Z}_{P^n} \slash H'_{\mu}$ too. So
      the subtorus $H \subset (\Z_2)^{k+n}$ that satisfies $\mathcal{Z}_{P^n} \slash H \cong Q^n$
      is not unique.
      \end{rem}
  \ \\
      \textbf{Proof of Proposition~\ref{prop:Partial-Quotient}:}
        Suppose $P^n$ has $k+n$ facets $F_1,\cdots, F_{k+n}$
        and $\pi_{\mu} : Q^n \rightarrow
        P^n$ is a small cover with a characteristic function $\mu$ on
        $P^n$. Let $V^n$ be a $\Z_2$-core of $Q^n$ with panels $\{ P_1, \cdots, P_k \}$
          described above.
        By Theorem~\ref{Thm:Equiv-Isom}, for any
         principal $(\Z_2)^m$-bundle $M^n$ over $Q^n$, there exists
        a $\lambda \in \mathrm{Col}_m(V^n)$ so that $M^n$ is equivalent to
        $M(V^n,\lambda)$ as principal $(\Z_2)^m$-bundles over $Q^n$.
        In addition, $M^n$ is connected implies that
        $L_{\lambda} = (\Z_2)^m$ (see Theorem~\ref{thm:comp}),
        and so $m\leq k$. Without loss of generality, suppose
        $\{ \lambda(P_1),\cdots, \lambda(P_m) \}$ is a linear basis of
        $L_{\lambda}$. In addition, we consider $(\Z_2)^m$ as a direct summand of
         $(\Z_2)^k$ and choose $\omega_1, \cdots, \omega_{k-m} \in (\Z_2)^k$ so
         that $(\Z_2)^k = L_{\lambda} \oplus \langle\omega_1 \rangle \oplus \cdots
        \oplus \langle\omega_{k-m}\rangle$.
       Let $\{ e_1, \cdots, e_k \}$ be a linear basis of $(\Z_2)^k$
       defined by the following:
       \begin{equation} \label{Equ:Basis}
        e_i = \lambda(P_i), \ 1\leq i \leq m ; \quad e_{m+j}
        = \omega_j, \ 1\leq j \leq k-m.
        \end{equation}

         As the above discussion, let $\lambda_0$ be a maximally independent
       $(\Z_2)^k$-coloring of $V^n$ defined by $\lambda_0(P_{i}) = e_{i}$ for
       $1\leq i \leq k$. And we let:
        $$N_{\lambda} := \langle e_{m+1},\cdots, e_{k} \rangle
        = \langle\omega_1 \rangle \oplus \cdots \oplus \langle\omega_{k-m}\rangle \subset
        (\Z_2)^k.$$
       Then we define an action $\widetilde{\eta}$ of $N_{\lambda}$ on $M(V^n,\lambda_0)$
         by: for any $1\leq j \leq k-m$,
          $$ \widetilde{\eta}(e_{m+j}) \left( \theta_{\lambda_0}(x,g) \right) :=
          (\lambda(P_{m+j})+\omega_{j}) \cdot
          \theta_{\lambda_0}(x,g) \overset{~\eqref{Equ:FreeAction-3}\ }{=}
          \theta_{\lambda_0}(x, g + \lambda(P_{m+j})+
         \omega_{j}).
         $$
        Obviously, the $N_{\lambda}$-action on $M(V^n,\lambda_0)$ defined by $\widetilde{\eta}$ is free.
        And
        by a parallel argument as in the proof of
        Lemma~\ref{Lem:Max_Indep}, we can show that the orbit space
        of this $N_{\lambda}$-action on $M(V^n,\lambda_0)$ is
        homeomorphic to $M(V^n,\lambda)$. Moreover, let:
         $$N^*_{\lambda} := \langle
          \lambda(P_{m+1})+\omega_1, \cdots, \lambda(P_k)+\omega_{k-m}
          \rangle \subset (\Z_2)^{k} .$$
        The rank of $N^*_{\lambda}$ is $k-m$. Obviously,
        the $N_{\lambda}$-action on $M(V^n,\lambda_0)$ defined by $\widetilde{\eta}$ is equivalent
        to the restriction of the natural $(\Z_2)^k$-action on $M(V^n,\lambda_0)$
        to $N^*_{\lambda}$. Then
        $M(V^n,\lambda)$ is homeomorphic to
        $M(V^n,\lambda_0) \slash N^*_{\lambda}$.
        Moreover, we can check that the natural action of $(\Z_2)^k\slash N^*_{\lambda}$ on
        $M(V^n,\lambda_0)\slash N^*_{\lambda}$
        is equivalent to the natural $(\Z_2)^{m}$-action on
        $M(V^n,\lambda)$. So $M^n \cong M(V^n,\lambda)$ is equivalent to
        $M(V^n,\lambda_0) \slash N^*_{\lambda}$ as principal
        $(\Z_2)^m$-bundles over $Q^n$.\vskip .2cm

        On the other hand, Corollary~\ref{Cor:Equivalence} says that
       the natural action of $(\Z_2)^k\slash N^*_{\lambda}$ on
        $ M(V^n,\lambda_0) \slash N^*_{\lambda}$ is equivalent to
        the canonical action of $H_{\mu} \slash \sigma(N^*_{\lambda})$ on
        $\mathcal{Z}_{P^n} \slash\sigma(N^*_{\lambda})$.
        Then combining all these equivalences, we have shown that $M^n$ is
        equivalent to the partial quotient $\mathcal{Z}_{P^n} \slash \sigma(N^*_{\lambda})$
        with the canonical $H_{\mu} \slash \sigma(N^*_{\lambda})$-action
        as principal $(\Z_2)^m$-bundles over $Q^n$.
        By the definition of $\{ e_1,\cdots, e_k \}$ in~\eqref{Equ:Basis} and the definition of
        $\sigma$ in~\eqref{Equ:Group-Isom}, $\sigma(e_{m+j}) = \sigma(\omega_j) = \omega_j +
        \mu(F_{m+j})$ for any $1\leq j \leq k-m$. So
        $\sigma(N^*_{\lambda})\subset H_{\mu} \subset (\Z_2)^{k+n}$ is generated
        by the set:
         $$\{  \sigma(\lambda(P_{m+1}))+\omega_1 + \mu(F_{m+1}),
        \cdots, \sigma(\lambda(P_k)) +\omega_{k-m} + \mu(F_k) \}.$$
        Notice that the choice for each $\omega_i$ is not unique, so
        the subtorus $H'$ of $(\Z_2)^{k+n}$ that satisfies
        $ \mathcal{Z}_{P^n} \slash H' \cong M^n $ is not unique either.
        \hfill $\square$

   \ \\

\end{document}